\documentclass[12pt]{article}

\usepackage{amssymb, amscd}
\usepackage{amsmath}
\usepackage{amssymb}
\usepackage{amsthm}
\usepackage{enumerate}
\usepackage[T1]{fontenc}
\usepackage{microtype}
\usepackage{libertine}
\usepackage[libertine]{newtxmath}
\usepackage[plainpages=false,hypertexnames=false,pdfpagelabels]{hyperref}

\theoremstyle{plain}
\newtheorem{prop}{Proposition}[section]

\newtheorem{thm}[prop]{Theorem}

\newtheorem{lem}[prop]{Lemma}

\theoremstyle{remark}
\newtheorem{dfn}[prop]{Definition}

\newtheorem{rmk}[prop]{Remark}
\newtheorem{ex}[prop]{Example}

\def\semicolon{;}
\def\applytolist#1{
    \expandafter\def\csname multi#1\endcsname##1{
        \def\multiack{##1}\ifx\multiack\semicolon
            \def\next{\relax}
        \else
            \csname #1\endcsname{##1}
            \def\next{\csname multi#1\endcsname}
        \fi
        \next}
    \csname multi#1\endcsname}

\makeatletter
\newcommand*{\transpose}{%
  {\mathpalette\@transpose{}}%
}
\newcommand*{\@transpose}[2]{%
  \raisebox{\depth}{$\m@th#1\intercal$}%
}
\makeatother

\newcommand{\floor}[1]{\left\lfloor#1\right\rfloor}
\newcommand{\twomat}[4]{\left(\begin{matrix}
  #1 & #2 \\ #3 & #4
    \end{matrix}
  \right)}
\newcommand{\stwomat}[4]{\left(\begin{smallmatrix}
  #1 & #2 \\ #3 & #4
    \end{smallmatrix}
  \right)}

\def\calc#1{\expandafter\def\csname c#1\endcsname{{\mathcal #1}}}
\applytolist{calc}QWERTYUIOPLKJHGFDSAZXCVBNM;
\def\bbc#1{\expandafter\def\csname bb#1\endcsname{{\mathbb #1}}}
\applytolist{bbc}QWERTYUIOPLKJHGFDSAZXCVBNM;
\def\bfc#1{\expandafter\def\csname bf#1\endcsname{{\mathbf #1}}}
\applytolist{bfc}QWERTYUIOPLKJHGFDSAZXCVBNM;

\def \pints{\bbZ_{(p)}}
\def \Zp{\bbZ_p}
\def \Qp{\bbQ_p}
\def \uhp{\cH}
\def \GL{\textbf{GL}}

\def \SL{\textbf{SL}}
\def \PSL{\textbf{PSL}}

\begin{document}
\title{On unbounded denominators and hypergeometric series}
\author{Cameron Franc, Terry Gannon, Geoffrey Mason\thanks{The first two authors were partially supported by grants from NSERC. The third author was supported by the Simons Foundation $\# 427007$.} \thanks{\href{mailto:franc@math.usask.ca}{franc@math.usask.ca}, \href{mailto:tgannon@math.ualberta.ca}{tgannon@math.ualberta.ca}, \href{mailto:gem@ucsc.edu}{gem@ucsc.edu}}}
\date{}
\maketitle

\abstract{We study the question of when the coefficients of a hypergeometric series are $p$-adically unbounded for a given rational prime $p$. Our first main result is a necessary and sufficient criterion (applicable to all but finitely many primes) for determining when the coefficients of a hypergeometric series with rational parameters are $p$-adically unbounded. This criterion is then used to show that the set of unbounded primes for a given series is, up to a finite discrepancy, a finite union of primes in arithmetic progressions. This set can be computed explicitly. We characterize when the density of the set of unbounded primes is $0$, and when it is $1$. Finally, we discuss the connection between this work and the unbounded denominators conjecture concerning Fourier coefficients of modular forms.}

\tableofcontents
\setcounter{tocdepth}{1}

\section{Introduction}
\label{s:intro}
Hypergeometric series are objects of considerable interest. From a number theoretic perspective, hypergeometric differential equations provide a convenient and explicit launching point for subjects such as $p$-adic differential equations \cite{Dwork1}, \cite{Dwork2}, rigid differential equations and Grothendieck's $p$-curvature conjecture \cite{Katz}, as well as the study of periods and motives. Recent attention has focused on the relationships between quotient singularities, integer ratios of factorials, and the Riemann hypothesis \cite{Borisov}, \cite{RodriguezVillegas},  \cite{Bober}.\ It is the Beukers-Heckman \cite{BeukersHeckman} classification of generalised hypergeometric differential equations with finite monodromy that underlies these connections.

If $F(z)$ is a solution of an ordinary differential equation (Fuchsian on $\bbP^1$, say) with a finite monodromy group, then it is an algebraic function. Moreover if $F(z)$ has rational Taylor coefficients, then an old theorem of Eisenstein states that for some integer $N$, the series $F(Nz)$ has integer coefficients, save for possibly the constant term (see \cite{DworkVanDerPoorten} for an interesting discussion of this result). This says two things:
\begin{enumerate}
\item $F(z)$ has $p$-adically bounded coefficients for almost all primes $p$;
\item for those primes for which $F(z)$ has $p$-adically unbounded coefficients, the coefficients cannot grow too quickly in $p$-adic absolute value.
\end{enumerate}
We say that $F(z)$ has $p$-adically unbounded coefficients when arbitrarily high powers of $p$ appear in the denominators of coefficients.
In the present paper,  given a hypergeometric series with rational coefficients, we study the set of all primes for which the series has $p$-adically unbounded coefficients. We do not assume that the monodromy is finite, although we do impose some mild restrictions, such as irreducibility of the monodromy representation. See the discussion at the start of Section \ref{s:hypergeometric} for a precise description of the conditions that we impose.

The basic tool that we use to study hypergeometric series is an old result of Kummer (cf.\ Theorem \ref{t:kummer}), characterising the $p$-adic valuation of binomial coefficients in terms of counting $p$-adic carries in certain $p$-adic additions.\ In Theorem \ref{t:coeffval} we employ Kummer's result to deduce a formula for the $p$-adic valuation of the coefficients of a generalized hypergeometric series $_nF_{n-1}$ with rational parameters.\ In Section \ref{s:n=2} we specialize to the classical case of $_2F_1$. The key result there is Theorem \ref{t:necandsuff}, which uses our valuation formula to give a convenient necessary and sufficient condition for  $_2F_1(a,b;c;z)$ (for generic rational parameters $a$, $b$ and $c$) to have $p$-adically unbounded coefficients. Here, generic means the fractional parts of $a,b,c$ are admissible in the sense of Definition \ref{d:admissible}. In the remainder of the section we study the set $S(a,b;c)$ of all primes for which a given
hypergeometric series $_2F_1(a,b;c;z)$ has $p$-adically unbounded coefficients. It turns out (Proposition \ref{p:SGP} and the following discussion) that, up to a finite discrepancy, the set $S(a,b;c)$ of unbounded primes is a union of all primes in a number of arithmetic progressions. The Dirichlet density of this set is an explicitly computable quantity. We characterize exactly when this density is zero (Theorem \ref{t:density0}), and also when the density is one (Theorem \ref{t:density1}). It turns out that the density is zero (for generic  parameters) precisely when the monodromy is finite. This result can be interpreted as a converse to Eisenstein's theorem, as it implies an infinite number of unbounded primes whenever the monodromy  is infinite. Note that the converse to Eisenstein's theorem only holds for hypergeometric series with generic parameters, as the famous example of $_2F_1(\frac 12, \frac 12;1;z)$ illustrates. At the other end of the spectrum, we show that $S(a,b;c)$ contains all but finitely many primes precisely when $c$ has the smallest fractional part of the three (generic) parameters. Said differently, one third of all basic hypergeometric series with rational parameters are as bad, from an arithmetic perspective, as possible.

Our interest in these questions comes from an old paper of Atkin--Swinnerton-Dyer \cite{ASD}, which raised the question of whether nonzero noncongruence scalar modular forms can have integer Fourier coefficients.\ The conjecture is that there will always be some unbounded primes\footnote{Since every scalar valued modular form on a finite index subgroup of $\SL_2(\bbZ)$ can be expressed as a power of the Dedekind $\eta$-function times an algebraic function of the classical $j$-invariant, Eisenstein's theorem implies there can only be finitely many such unbounded primes if the noncongruence subgroup is of finite index in $\SL_2(\bbZ)$.}.\ Some of the most interesting work on this unbounded denominators conjecture (UBD) is due to Anthony Scholl \cite{AS} and Winnie Li and  Ling Long together with their collaborators and students \cite{LL}, \cite{KL}.\ In \cite{FrancMason1} and \cite{Gannon}, the authors  showed explicitly how to describe vector-valued modular forms of rank two for $\SL_2(\bbZ)$ in terms of hypergeometric series $_2F_1$, and in \cite{FrancMason1} these series were then used to verify the extension of the UBD conjecture to vector-valued modular forms of rank two for $\SL_2(\bbZ)$.\ Shortly after, a similar result was proved \cite{FrancMason2} for some vector-valued modular forms of rank three, using generalized hypergeometric series $_3F_2$. In both of these papers, only a very conservative use of hypergeometric series was made. For example, in the case of series with infinite monodromy, unbounded denominators were established in \cite{FrancMason1} by showing that there exists an arithmetic progression of primes $p$ that occur at least to power $p^{-1}$ in the coefficients of the given series. But it was unclear whether the coefficients were in fact $p$-adically unbounded for such primes. The present paper gives rather complete answers to the questions raised in the final sections of \cite{FrancMason1} and \cite{MasonFcoeffs}. In particular, we prove (for generic parameters) that if the $\SL_2(\bbZ)$ representation has infinite image, then there is a positive density of unbounded primes (if the image is finite, the kernel is conguence and all primes will be $p$-adically bounded). In Section \ref{s:Schwarzlist} we recall some of these facts relating modular forms and hypergeometric series, and we combine the results of \cite{Mason2-dim} and \cite{FrancMason1} to enumerate a modular analogue at level one of the famous Schwarz list classifying hypergeometric series with finite monodromy groups. Applying our results to these examples yields an independent (and completely elementary) verification that those modular forms have bounded denominators.

More interesting is the application of our results to the UBD conjecture for  vector-valued modular forms of rank two for $\Gamma(2)$.  This will be a more serious test of the UBD conjecture: there are three parameters worth of 2-dimensional $\Gamma(2)$-representations ($\SL_2(\bbZ)$ has requires only one parameter), and almost every finite image 2-dimensional $\Gamma(2)$-representation is noncongruence, whereas no $\SL_2(\bbZ)$ ones are. Richard Gottesman treats precisely this question  in his upcoming PhD thesis \cite{Gottesman}, building on the results of this paper.

\section{Basic notions and notations}
\label{s:basic}
Throughout $p$ denotes a rational prime, and $\pints$ denotes the ring of rational numbers $\frac ab$ where $p$ does not divide $b$. Hence $\pints^\times$ denotes the set of rational numbers $\frac ab$ such that $p$ is coprime to both $a$ and $b$. Let $\Zp$  denote the ring of $p$-adic integers, which is the completion of $\pints$ for the $p$-adic valuation $v_p$ normalized so that $v_p(p) = 1$. Thus, for rational $r$, $v_p(r)$ denotes the exact power of $p$ occurring in the prime decomposition of $r$. Let $\Qp$ denote the field of $p$-adic numbers. We recall the following basic facts concerning $p$-adic expansions.
\begin{lem}
\label{l:periodicexpansion}
With notation as above:
\begin{enumerate}
\item An element $x \in \pints^\times$ has a purely periodic $p$-adic expansion if and only if $x \in [-1,0)$.
\item Let $\frac nd \in \pints^\times$ have a purely periodic $p$-adic expansion, of minimal period $M$. Assume that $\gcd(n,d) = 1$. Then $M$ is the multiplicative order of $p$ in $(\bbZ/d\bbZ)^\times$.
\end{enumerate}
\end{lem}
\begin{proof}
First let $x = c_0c_1\cdots c_{M-1}c_0c_1\cdots$ be the periodic $p$-adic expansion of $x$. If $y$ denotes the integer whose expanion in base $p$ is $y = c_0c_1\cdots c_{M-1}$, then $x = \frac{y}{1-p^M}$. Since $0 < y \leq p^M-1$, it follows that $x \in [-1,0)$. Observe that if $x = \frac nd$ with $\gcd(n,d) = 1$, then this shows that $d$ divides $p^M-1$, so that the period $M$ is at least as large as the order of $p$ in $(\bbZ/d\bbZ)^\times$.

Conversely, suppose $\frac nd \in \pints^\times \cap [-1,0)$, with $\gcd(n,d) = 1$, and let $M$ be the order of $p$ in $(\bbZ/d\bbZ)^\times$. Assume $n > 0$ and $d< 0$. Then $1-p^M = du$ for a positive integer $u$, and thus $\frac{n}{d}=\frac{nu}{1-p^M}$. Observe that $0 < nu \leq p^M-1$, so that the positive integer $nu$ has a finite $p$-adic expansion of at most $M$ digits. But then $\frac{nu}{1-p^M}$ visibly has a periodic $p$-adic expansion of period dividing $M$. Since the order $M$ of $p$ modulo $d$ was seen above to be a lower bound for the minimal possible period, it follows that $M$ is indeed the minimal period of the $p$-adic expansion of $\frac nd$.
\end{proof}

If $x$ is a real number, then let $\floor{x}$ denote the unique integer satisfying $\floor{x} \leq x < \floor{x}+1$. Similarly define the fractional part of $x$ by $\{x\} = x - \floor{x}$.

\begin{dfn}
\label{d:truncation}
If $a \in \bbZ_p$ is a $p$-adic integer, then for each $j \geq 0$ let $\tau_j(a)$ denote the unique integer satisfying $0 \leq \tau_j(a) < p^j$ such that $\tau_j(a) \equiv a \pmod{p^j}$. The maps $\tau_j \colon \bbZ_p \to \bbZ$ are called \emph{truncation} operators.
\end{dfn}

\begin{lem}
\label{l:digits}
Let $x=\frac{n}{d}$ denote a rational number with $\gcd(n,d) = 1$ satisfying $0<  x < 1$, and let $p$ denote a prime such that $x-1 \in \Zp^\times$. Let $M$ denote the order of $p$ mod $d$, and let $x-1 = \overline{x_0x_1\ldots x_{M-1}}$ denote the $p$-adic expansion of $x-1$. Then for each index $0 \leq j < M$,
\[
  x_j = \floor{\left\{-p^{M-1-j}x\right\}p}.
\]
\end{lem}
\begin{proof}
Our hypotheses on $x$ ensure that both $-x$ and $x-1$ have periodic $p$-adic expansions. Let $-p^{M-1-j}n = \alpha d + r$ where $0 \leq r < d$. The $p$-adic expansion of $-p^{M-1-j}x = \alpha+1 + \left(\frac{r}{d}-1\right)$ is
\[
  -p^{M-1-j}x = \overbrace{00\cdots 0}^{M-1-j \textrm{ terms}}\overline{(p-1-x_0)(p-1-x_1)\cdots (p-1-x_{M-1})}
\]
Now, $\alpha+1$ is uniquely determined as that integer such that when you subtract it from this expansion, you get a purely periodic expansion. Hence $-\alpha-1 = (p-1-x_{j+1})(p-1-x_{j+2})\cdots (p-1-x_{M-1})$ and
\[
  \frac{r}{d}-1 = \overline{(p-1-x_{j+1})\cdots (p-1-x_{M-1})(p-1-x_0)\cdots (p-1-x_j)}.
\]
Observe that $\tau_M(\frac rd-1) = (1-\frac rd)(p^M-1)$ and hence
\[
  0 \leq \left(1-\frac rd\right)(p^M-1) - (p-1-x_j)p^{M-1} \leq p^{M-1}-1.
\]
This is equivalent with
\[
  0 \leq \frac{x_j}{p^M-1} \leq \frac rdp - x_j \leq \frac{p^{M}+x_j-p}{p^M-1} \leq 1.
\]
Since $\left\{-p^{M-1-j}x\right\} = \frac rd$, the proof is complete unless there is equality on the right above. But equality can only occur if $x_j = p-1$ and $\frac{r}{d}p-x_j = 1$, that is, $\frac{r}{d} = 1$. Since $\frac{r}{d} < 1$, this concludes the proof.
\end{proof}

Let $x = \frac nd$ satisfy $0 < x < 1$, and given a prime $p$, let $x_j(p)$ denote the $j$th $p$-adic digit of $x-1$. Observe that Lemma \ref{l:digits} implies that if $p$ varies over primes in a fixed residue class $p \equiv r \pmod{d}$, then the ratio $\frac{x_j(p)}{p}$ converges to the quantity $\{-r^{M-1-j}x\}$, which depends on the residue class $r \pmod{d}$, but not on $p$. This fact will explain why if a hypergeometric series has $p$-adically bounded or unbounded coefficients for large enough primes, then it will be similarly $p$-adically bounded or unbounded for all large enough primes in a corresponding congruence class.

\begin{lem}
\label{l:distinctdigits}
Let $x$ and $y$ denote rational numbers satisfying $0 < x < y < 1$, and such that $x-1,y-1 \in \Zp^\times$ for a rational prime $p$. Let $x-1 = \sum_{j \geq 0} x_j(p)p^j$ denote the $p$-adic expansion of $x-1$, and define the digits $y_j(p)$ similarly. Let $D$ denote the least common multiple of the denominators of $x$ and $y$. Then if $p > D$, one has $x_j(p) \neq y_j(p)$ for all $j \geq 0$.
\end{lem}
\begin{proof}
Observe that if $p > D$ then both $x$ and $y$ satisfy the hypotheses of Lemma \ref{l:digits}. Thus, the condition $x_j(p) = y_j(p)$ is equivalent with
\begin{equation}
\label{eq:samedigit}
\floor{\left\{-p^{M-1-j}x\right\}p}=\floor{\left\{-p^{M-1-j}y\right\}p}.
\end{equation}
Let $\frac{\alpha}{D}$ and $\frac{\beta}{D}$ denote the fractional parts above, and without loss of generality take $\alpha \leq \beta$. Then \eqref{eq:samedigit} is equivalent with $0 \leq (\beta-\alpha) < \frac{D}{p}$. Hence if $p >D$ then \eqref{eq:samedigit} is equivalent with the simpler identity $\left\{-p^{M-1-j}x\right\} = \left\{-p^{M-1-j}y\right\}$. But this forces $p^{M-1-j}(x-y)$ to be an integer, contradicting the fact that $p$ is coprime to $D$. Thus, if $p > D$ then $x_j(p) \neq y_j(p)$.
\end{proof}

If $a$ and $b$ are $p$-adic numbers, then let $c_p(a,b)$ denote the number of $p$-adic carries that are required to evaluate the sum $a+b$. Hence $c_p$ defines a map $c_p \colon \Qp^2 \to \bbN\cup \{\infty\}$. Recall that the binomial polynomials
\[
  \binom{x+n}{n} = \frac{(x+1)(x+2)\cdots (x+n)}{n!}
\]
define continuous functions on $\Zp$ that vanish only at $-n,-n+1,\ldots,-1$.
\begin{thm}[Kummer]
\label{t:kummer}
Let $x \in \Zp$ and $n \in \bbZ_{\geq 0}$. Then $v_p\binom {x+n}n = c_p(x,n)$.
\end{thm}
\begin{proof}
This was proved by Kummer in \cite{Kummer} when $x \in \bbZ_{\geq 0}$. A uniform proof that handles all $p$-adic integers can be given, but we will show that the result for integral $x$ extends to $p$-adic integers by continuity. 

Since each polynomial $\binom{x+n}{n}$ for $n \in \bbZ_{\geq 0}$ is continuous on $\Zp$, and $v_p$ is continuous on $\Qp^\times$, the left side of the claimed equality is continuous on $\Zp \setminus \{-n,-n+1,\ldots, -1\}$. 

We claim that $c_p(x,n) = \infty$ if and only if $x \in \{-n,-n+1,\ldots, -1\}$. First, if $x \in \{-n,-n+1,\ldots, -1\}$ then the $p$-adic expansion of $x$ contains infinitely many nonzero digits. On the other hand, $x+n$ is an integer such that $x+n\geq 0$, so that it has a finite $p$-adic expansion (similarly for $n$). This implies that an infinite number of carries had to occur to evaluate $x+n$, meaning $c_p(x,n) = \infty$ as claimed.

On the other hand, suppose that $c_p(x,n) = \infty$. Since $n$ has a finite $p$-adic expansion, the only way this can occur is if eventually the $p$-adic expansion of $x$ is $p-1$ repeated infinitely often. That is, $x = y -p^N$ for some integer $0 \leq y < p^N$ and $N \geq 1$. In particular, $x$ is a strictly negative integer. Now the condition that $c_p(x,n) = \infty$ is equivalent to $y+n \geq p^N$, hence $x \geq -n$. 

This verifies that $c_p(x,n)$ takes finite values on $\Zp\setminus\{-n,-n+1,\ldots, -1\}$. It is easily seen to be continuous there, since if $c_p(x,n) < \infty$, then the evaluation of $x+n$ will not involve carries beyond a digit corresponding to some large power $p^N$. But then $c_p(x+\alpha p^{N+1},n) = c_p(x,n)$ for $\alpha \in \Zp$ shows that $c_p(x,n)$ is $p$-adically continuous (in fact locally constant) on $\Zp\setminus \{-n,-n+1\ldots, -1\}$. Hence Kummer's theorem extends to $x \in \Zp$ by continuity.
\end{proof}

\begin{rmk}
Kummer's Theorem \ref{t:kummer} does not extend to $x \in \bbQ_p$. For example, if $n \in \bbZ_{\geq 0}$ then $c_p(\frac 1p,n) = 0$, while $v_p\binom{\frac 1p+n}{n} = -n-v_p(n!) < 0$. This observation is what prevents us from handling primes that divide the denominators of hypergeometric parameters in a uniform manner with all other primes.
\end{rmk}

\section{Hypergeometric series}
\label{s:hypergeometric}
Recall that for integers $n \geq 1$, the \emph{generalized hypergeometric series} $_nF_{n-1}$ is defined by the formula
\[
  _nF_{n-1}(\alpha_1,\ldots,\alpha_n; \beta_1,\ldots, \beta_{n-1};z) = \sum_{m\geq 0}\frac{\prod_{j=1}^n(\alpha_j)_m}{\prod_{k=1}^{n-1}(\beta_k)_m} \frac{z^m}{m!}.
\]
Here $(a)_m$ denotes the Pochhammer symbol defined by 
\[(a)_m = \begin{cases}
1 & m = 0,\\
a(a+1)\cdots(a+m-1) &m \geq 1.
\end{cases}\]
Let $A_m$ denote the $m$th coefficient of $_nF_{n-1}(\alpha_j; \beta_k;z)$. Observe that each $A_m$ is a rational number provided that the hypergeometric parameters $\alpha_j$ and $\beta_k$ are rational. In this note we only consider rational parameters.

Hypergeometric series are solutions of Fuchsian differential equations. We restrict here to series with rational coefficients, so we require the parameters $\alpha_j,\beta_k$ to be rational. If $\alpha_j - \beta_k \in \bbZ$ for some indices $j$ and $k$, then the corresponding monodromy representation is reducible (Proposition 2.7 of \cite{BeukersHeckman}).   If some $\alpha_j$ lies in $\bbZ_{< 0}$, then $_nF_{n-1}(\alpha_j;\beta_k;z)$ is a polynomial. Obviously no $\beta_k$ can lie in $\bbZ_{<0}$ for $_nF_{n-1}$ to be well-defined, and if $\beta_k \in \bbZ_{\geq 0}$ then the monodromy around $0$ will have repeated eigenvalues. Finally, provided none of $\alpha_j,\beta_k,\alpha_j-\beta_k$ are integers, it can be shown that  the series $_nF_{n-1}(\alpha_j;\beta_k;z)$ is $p$-adically bounded iff $_nF_{n-1}(\alpha_j+m_j;\beta_k+n_k;z)$ is, for any $m_j,n_k\in\bbZ$  (the proof for $_2F_1$ is given in Lemma \ref{l:admissible}; the proof for general $_nF_{n-1}$ is given in \cite{Bernstein}). Moreover, if the monodromy is irreducible, then shifting the hypergeometric parameters by integers also does not affect the monodromy (Corollary 2.6 of \cite{BeukersHeckman}). Thus we arrive at the following definition:
\begin{dfn}
\label{d:admissible}
Rational hypergeometric parameters $(\alpha_1,\ldots, \alpha_n;\beta_1,\ldots, \beta_{n-1})$ are said to be \emph{admissible} provided that the following two conditions are satisfied:
\begin{enumerate}
\item $0 < \alpha_j,\beta_k < 1$ for all $j$ and $k$;
\item $\alpha_j \neq \beta_k$ for all $j$ and $k$.
\end{enumerate}
\end{dfn}

\begin{dfn}
\label{d:goodprime}
Let $(\alpha_1,\ldots, \alpha_n; \beta_1,\ldots, \beta_{n-1})$ denote rational hypergeometric parameters. A rational prime $p$ is \emph{good} for these parameters provided that $v_p(\alpha_j-1) = v_p(\beta_k-1) = 0$ for all $j$ and $k$.
\end{dfn}
Note that for fixed parameters, all but finitely many primes are good. For almost all choices of rationals
$\alpha_j,\beta_k$, the fractional parts $(\{\alpha_j\},\{\beta_k\})$ are admissible and have the identical list of bounded primes as $(\alpha_j,\beta_k)$.

\begin{dfn}
\label{d:period}
Let $(\alpha_1,\ldots, \alpha_n; \beta_1,\ldots, \beta_{n-1})$ denote rational hypergeometric parameters, and assume that $p$ is a good prime. The associated \emph{period} is the least common multiple of the multiplicative order of $p$ modulo the various (reduced) denominators of the quantities $\alpha_j-1$ and $\beta_k-1$.
\end{dfn}

Good primes associated with admissible parameters are precisely the primes such that the quantities $\alpha_j-1$ and $\beta_k-1$ have purely periodic $p$-adic expansions. The corresponding period is then nothing but the least common multiple of the various periods of these expansions.

\begin{thm}
\label{t:coeffval}
Let $(\alpha_1,\ldots,\alpha_n; \beta_1,\ldots, \beta_{n-1})$ denote rational hypergeometric parameters, and let $p$ denote a prime such that $v_p(\alpha_j-1) \geq 0$ and $v_p(\beta_k-1) \geq 0$ for all $j$ and $k$. Then if $A_m$ denotes the $m$th coefficient of $_nF_{n-1}(\alpha_j;\beta_k;z)$,
\begin{equation}
\label{eq:valuation}
 v_p(A_m) = \sum_{j=1}^n c_p(\alpha_j-1,m) - \sum_{k=1}^{n-1} c_p(\beta_k-1,m).
\end{equation}
Further, assume that the parameters $(\alpha_1,\ldots, \alpha_n;\beta_1,\ldots, \beta_{n-1})$ are admissible, assume that $p$ is a good prime for this data, and let $M$ denote the corresponding period. Then
\begin{equation}
\label{eq:periodicity}
  v_p(A_{mp^M}) = v_p(A_m)
\end{equation}
for all $m \in \bbZ_{\geq 0}$.
\end{thm}
\begin{proof}
First observe that
\[
  A_m = \frac{\prod_{j=1}^n(\alpha_j)_m}{\prod_{k=1}^{n-1}(\beta_k)_m} \frac{1}{m!} = \frac{\prod_{j=1}^{n}\binom{\alpha_j-1+m}{m}}{\prod_{k=1}^{n-1}\binom{\beta_k-1+m}{m}}
\]
Thus \eqref{eq:valuation} follows immediately from Theorem \ref{t:kummer}.

For the next claim, observe that since we have assumed that $p$ is good, each of $\alpha_j-1$ and $\beta_k-1$ has a periodic $p$-adic expansions of period dividing $M$, by Lemma \ref{l:periodicexpansion}. Hence $c_p(\alpha_j-1,mp^M) = c_p(\alpha_j-1,m)$ and $c_p(\beta_k-1,mp^M) = c_p(\beta_k-1,m)$, and thus \eqref{eq:periodicity} follows from \eqref{eq:valuation}.
\end{proof}

\begin{rmk}
It follows immediately from Theorem \ref{t:coeffval} that under the hypotheses of that theorem,
\[
  -(n-1)\log_p(m) \leq v_p(A_m) \leq n\log_p(m).
\]
\end{rmk}

\begin{ex}
Consider the hypergeometric series $_2F_1(\frac 12, \frac 12;1;z)$ that arises in the study of the monodromy of the Legendre family of elliptic curves. Theorem \ref{t:kummer} immediately implies the well-known fact that $v_p(A_m) = 2c_p(-\frac 12,m) \geq 0$ for all odd primes $p$. This is an example of a solution of a differential equation with an infinite monodromy group, but for which there is a a unique prime such that the coefficients of $_2F_1(\frac 12, \frac 12;1;z)$ are $p$-adically unbounded (obviously $p =2$ is the bad prime). This does not contradict Theorem \ref{t:density0} below, as the parameters $(\frac 12,\frac 12;1)$ are not admissible.
\end{ex}

\begin{prop}
\label{p:bigsup}
Let $(\alpha_1,\ldots, \alpha_n;\beta_1,\ldots, \beta_{n-1})$ denote admissible parameters, and let $p$ denote a corresponding good prime. Then
\[
  \sup_m v_p(A_m) = \infty.
\]
In particular, for any given set of admissible hypergeometric parameters, we have $\sup_m v_p(A_m) = \infty$ for all but finitely many primes $p$.
\end{prop}
\begin{proof}
Let $a$ be any nonzero $p$-adic integer with a purely periodic expansion of period $M$, and let $N_r = p^{Mr}-p^{M(r-1)}-\cdots-p^M-1$ be the integer whose $p$-adic expansion is $M$ copies of $(p-1)$, and then this is followed by $r-1$ segments of the form $(p-2)(p-1)\cdots (p-1)$. Some digit in each segment of size $M$ of the coefficients of $a$ is nonzero. In fact, the first such digit is indexed by $v_p(a)$. Since there are no carries before this digit, the number of carries from computing $a+N_r$ in the first segment of size $M$ is $M-v_p(a)$, and similarly for every other segment. Thus $c_p(a,N_r) = (M-v_p(a))r$. Taking $a$ to be each of $\alpha_j-1$ and $\beta_k-1$, it follows by Theorem \ref{t:coeffval} that
\[
  v_p(A_{N_r}) = \left(M-\sum_{j=1}^n v_p(\alpha_j-1) + \sum_{k=1}^{n-1}v_p(\beta_k-1)\right)r = Mr,
\]
where the last equality uses the assumption that $p$ is a good prime. Thus $v_p(A_{N_r})$ is unbounded, which proves the Proposition.
\end{proof}
\begin{rmk}
The proof of Proposition \ref{p:bigsup} shows that under those hypotheses, the sequence $v_p(A_m)$ has a subsequence that diverges like $\log_p(m)$.
\end{rmk}

\section{The case of $_2F_1$}
\label{s:n=2}
We are more interested in characterising when $\inf_m v_p(A_m) = -\infty$, that is, when $_nF_{n-1}$ has $p$-adically unbounded coefficients. It will be convenient to specialize to the case $n = 2$. First we show, using an argument that goes back to Gauss, that there is no loss in generality when considering only admissible parameters.
\begin{lem}
\label{l:admissible}
Let $(a,b;c)$ and $(r,s;t)$ denote two sets of rational hypergeometric parameters, and assume that 
\begin{enumerate}
\item[(i)] none of $a$, $b$, $c$, $a-c$ or $b-c$  is an integer;
\item[(ii)] $a-r$, $b-s$ and $c-t$ are all integers.
\end{enumerate}
Then for each prime $p$, the series $_2F_1(a,b;c;z)$ has $p$-adically unbounded coefficients if and only if the same is true for $_2F_1(r,s;t;z)$.
\end{lem}
\begin{proof}
Let $\theta = z\frac{d}{dz}$. Then one easily verifies the following identities, which go back to Gauss (see section I.1 of \cite{Matsuda}):
\begin{align*}
_2F_1(a+1,b;c;z) &=\left(1+\frac 1a\theta\right) {}_2F_1(a,b;c;z),\\
_2F_1(a-1,b;c;z) &=\left((1-z)-\frac{(a+b-c)z}{c-a} + \frac{1-z}{c-a}\theta\right) {}_2F_1(a,b;c;z),\\
_2F_1(a,b;c+1;z) &=\left(\frac{(a+b-c)c}{(c-a)(c-b)}+\frac{(1-z)c}{(c-a)(c-b)}\frac{d}{dz}\right) {}_2F_1(a,b;c;z),\\
_2F_1(a,b;c-1;z) &=\left(1+\frac{1}{c-1}\theta\right) {}_2F_1(a,b;c;z).
\end{align*}
From this one sees that the claim holds if $(r,s;t) = (a\pm 1,b;c)$ or $(r,s;t) = (a,b;c\pm 1)$. The general case then follows by symmetry and repeated application of the cases already treated.
\end{proof}

Next we establish a necesary and sufficient condition for a hypergeometric series to have $p$-adically unbounded coefficients for all good primes $p$. Recall that if $x$ is a $p$-adic integer, then $\tau_j(x)$ denotes the truncation of $x$ mod $p^j$.
\begin{thm}
\label{t:necandsuff}
Let $(a,b;c)$ denote admissible hypergeometric parameters, and let $p$ denote a good prime. Then the following are equivalent:
\begin{enumerate}
\item[(i)] for some index $j$, we have $\tau_j(c-1) > \tau_j(a-1)$ and $\tau_j(c-1) > \tau_j(b-1)$;
\item[(ii)] $_2F_1(a,b;c;z)$ has $p$-adically unbounded coefficients.
\end{enumerate}
For good primes $p$, if the coefficients of $_2F_1(a,b;c;z)$ are $p$-adically bounded, then they are in fact $p$-adic integers.
\end{thm}
\begin{proof}
Let $a_j$, $b_j$ and $c_j$ denote the digits in the $p$-adic expansions of $a-1$, $b-1$ and $c-1$, respectively. First assume that (i) holds, and let $j$ denote the smallest index such that $\tau_j(c-1) > \tau_j(a-1),\tau_j(b-1)$. By minimality of $j$ we can't have $c_{j-1}=a_{j-1}=b_{j-1}$. Without loss of generality we may assume that one of the following two conditions holds:
\begin{enumerate}
\item[(a)] $c_{j-1} > a_{j-1}$ and $c_{j-1} > b_{j-1}$;
\item[(b)] $c_{j-1} = a_{j-1}$ but $c_{j-1} > b_{j-1}$.
\end{enumerate}
We will define two different subsequences $m_r$, one for each case, such that $v_p(A_{m_r})$ diverges to $-\infty$.

In case (a), let $m_r = \sum_{s=0}^r(p-c_{j-1})p^{Ms+j-1}$ where $M$ is the period of the data $(a,b;c)$ and $p$. Observe that $c_{j-1}$ is nonzero, since it is strictly larger than $a_{j-1}$, say. Hence the expression defining $m_r$ is its $p$-adic expansion. We have $c_p(c-1,m_r) \geq r+1$ while $c_p(a-1,m_r) = c_p(b-1,m_r) = 0$. Thus by Theorem \ref{t:coeffval} we have $v_p(A_{m_r}) \leq -r-1$, which shows that the coefficients $A_m$ are $p$-adically unbounded in case (a).

In case (b), since $c_{j-1}=a_{j-1}$ while $\tau_j(c-1) > \tau_j(a-1)$, there exists a string of digits where $c_{i-1}=a_{i-1}$ for $i$ in some range $k < i \leq j$, but then $c_{k-1} > a_{k-1}$. Let
\[
  m_r = \sum_{s=0}^r\left((p-c_{k-1})p^{k-1} + \sum_{i=k}^{j-1}(p-c_{i}-1)p^i\right )p^{Ms}.
\]
Again, in this case $c_{k-1} > a_{k-1} \geq 0$, which shows that $c_{k-1}$ is nonzero and the expression defining $m_r$ is its $p$-adic expansion. Clearly $c_p(c-1,n_r) \geq (r+1)(j-k)$ and $c_p(a-1,m_r) = 0$. Similarly, since $c_{j-1} > b_{j-1}$, there can be no carry at the $j$th digit when evaluating $b-1+m_r$, and thus $c_p(b-1,m_r) \leq (r+1)(j-k-1)$. Hence by Theorem \ref{t:coeffval} we have 
\[v_p(A_{m_r}) \leq (r+1)(j-k-1)-(r+1)(j-k) = -r-1.\] 
This shows that in case (b), the coefficients $A_m$ are $p$-adically unbounded.

Conversely, assume that (i) does not hold. That is, assume that for every index $j$, we have
\[
  \tau_j(c-1) \leq \max(\tau_j(a-1),\tau_j(b-1)).
\]
We will show that if $\tau_j(c-1) \leq \tau_j(a-1)$, and if there is a $p$-adic carry at the $j$th digit when evaluating $(c-1)+m$, then there is also a carry at the $j$th digit when evaluating $(a-1)+m$. If $\tau_j(c-1) = \tau_j(a-1)$, this is obvious. 

We may thus assume that $\tau_j(c-1) < \tau_j(a-1)$. It follows that there is an index $0 \leq s \leq j-1$ such that $c_s < a_s$ but $c_k = a_k$ for $s < k \leq j-1$ (if any such indices $k$ exist). Let $m = m_0m_1\cdots$. Since there is a carry at digit $j$ in $(c-1)+m$, there are two possibilities: 
\begin{enumerate}
\item[(a)] $m_{j-1} \geq (p-c_{j-1})$ and the carry did not rely on an earlier carry;
\item[(b)] $m_{j-1} = (p-c_{j-1}-1)$ and the carry only occured because it was preceded by an earlier carry.
\end{enumerate}
If $m_{j-1} \geq (p-c_{j-1})$, then there is clearly also a carry at digit $j$ in $(a-1)+m$. So it remain to treat case (b).

Suppose that $m_{j-1} = (p-c_{j-1}-1)$, so that the $j$ carry in $(c-1)+m$ implies that there is an earlier carry. If the earlier carries do not extend back past digit $s$, there is some $m_k$ in that range such that $c_k \neq 0$ and $m_k \geq (p-c_k)$. But then in this range $a_k \geq c_k$ and there must alo be a sequence of carries in $(a-1)+m$ from the $k$th digit up to the $j$th, forcing a carry at digit $j$ as claimed. If the earlier carries in $(c-1)+m$ go even further to the left, past digit $s$, then we still win since $a_s > c_s$. Hence even if $m_s = (p-c_s-1)$, there is a carry at digit $s$ in $(a-1)+m$, and this will then force a sequence of carries up to the $j$th digit.

This verifies the claim that if $\tau_j(c-1) \leq \tau_j(a-1)$ and there is a carry at digit $j$ in $(c-1)+m$, then there is also a carry at digit $j$ in $(a-1)+m$. By symmetry we see that necessarily $c_p(c-1,m) \leq c_p(a-1,m)+c_p(b-1,m)$ for all integers $m \geq 0$. Hence $v_p(A_m) \geq 0$ by Theorem \ref{t:coeffval}, and this concludes the proof of the Theorem.
\end{proof}

\begin{rmk}
The proof of Theorem \ref{t:necandsuff} shows that under those hypotheses, if the coefficients of $_2F_1$ are $p$-adically unbounded, then the sequence $v_p(A_m)$ has a subsequence that diverges to $-\infty$ at least as quickly as $\frac{1}{M}\log_p(m)$.
\end{rmk}

\begin{rmk}
If $p$ is larger than the least common multiple of the denominators of $a$, $b$ and $c$, then it suffices, by Lemma \ref{l:digits}, to compare the $p$-adic digits of $a-1$, $b-1$ and $c-1$ in Theorem \ref{t:necandsuff}, rather than their truncations.
\end{rmk}

\begin{rmk}
  Theorem \ref{t:necandsuff} omits consideration of the finite number of primes that are not good for a given set of parameters. Recall that this means that for at least one of the parameters $x$, either $v_p(x-1) > 0$ or $v_p(x-1) < 0$. In the first case, we can write $x-1 = p^re$ for some other rational number $e$ in $(-1,0)$ that is coprime to $p$, and which thus has a periodic $p$-adic expansion. In this case $x-1$ has an expansion beginning with $r$ zeros, and then it becomes periodic. Theorem \ref{t:kummer} can still be used to analyze such primes as in our proof of Theorem \ref{t:necandsuff}. The case where $v_p(x-1) < 0$ is even easier, as then one can write $x = p^{-r}\frac nd$ for some $r \geq 1$ and $n,d\in\bbZ$ coprime to $p$. Thus
  \[
    (x)_m = p^{-rm}d^{-m}n(n+p^rd)(n+2p^rd)\cdots(n+(m-1)p^rd),
  \]
and $v_p((x)_m) = -rm$. However, since there are three parameters to consider, and it becomes cumbersome to formulate a definitive result for all cases that arise, we opted to state Theorem \ref{t:necandsuff} only for good primes.
\end{rmk}

\begin{ex}
Consider the admissible parameters $a = 1/6$, $b = 5/6$ and $c = 1/5$. They correspond to a hypergeometric equation with a finite monodromy group, and hence there should be only finitely many primes $p$ such that $_2F_1\left(\frac 16,\frac 56;\frac 15;z\right)$ has $p$-adically unbounded coefficients. 

Let $p \geq 7$ be a prime. Then by Lemma \ref{l:digits},
\begin{align*}
a-1=-\frac 56 &= \begin{cases}
\overline{\left(\frac{5p-5}{6}\right)} & p\equiv 1 \pmod{6},\\
\overline{\left(\frac{p-5}{6}\right)\left(\frac{5p-1}{6}\right)} &p\equiv 5 \pmod{6},
\end{cases}\\
b-1=-\frac 16 &= \begin{cases}
\overline{\left(\frac{p-1}{6}\right)} & p\equiv 1 \pmod{6},\\
\overline{\left(\frac{5p-1}{6}\right)\left(\frac{p-5}{6}\right)}& p\equiv 5 \pmod{6},
\end{cases}\\
c-1=-\frac 45&= \begin{cases}
\overline{\left(\frac{4p-4}{5}\right)}& p\equiv 1 \pmod{5},\\
\overline{\left(\frac{p-4}{5}\right)\left(\frac{4p-1}{5}\right)}& p\equiv 4 \pmod{5},\\
\overline{\left(\frac{2p-4}{5}\right)\left(\frac{p-2}{5}\right)\left(\frac{3p-1}{5}\right)\left(\frac{4p-3}{5}\right)}& p\equiv 2 \pmod{5},\\
\overline{\left(\frac{3p-4}{5}\right)\left(\frac{p-3}{5}\right)\left(\frac{2p-1}{5}\right)\left(\frac{4p-2}{5}\right)}& p\equiv 3 \pmod{5}.
\end{cases}
\end{align*}
There are thus eight cases to consider, and it is straightforward to use Theorem \ref{t:necandsuff} to check that $_2F_1(\frac 16,\frac 56;\frac 15;z)$ is $p$-integral in each of them. For example, suppose that $p\equiv 7 \pmod{30}$. We see that $a-1$ and $b-1$ are both $1$-periodic and $\max\{\tau_j(a-1),\tau_j(b-1)\} = \tau_j(a-1)$ for all $j$. As long as $p > 7$, then $\frac{5p-5}{6}$ is larger than each $p$-adic digit of $c-1$, and hence $_2F_1(\frac 16,\frac 56;\frac 15;z)$ is $p$-integral for such primes. If $p = 7$ then we have
\begin{align*}
a-1 &= \overline{5555}, & b-1 &= \overline{1111}, & c-1 &= \overline{2145}.
\end{align*}
In this case we still have $7$-integrality since, in terms of $7$-adic expansions, $2 \leq 5$, $21 \leq 55$, $214 \leq 555$ and $2145 \leq 5555$. Hence if $p \equiv 7 \pmod{30}$, then $_2F_1(\frac 16,\frac 56;\frac 15;z)$ is $p$-integral, as was claimed. The other seven cases are similar.
\end{ex}

\begin{ex}
Next consider $a = \frac 15$, $b = \frac 13$ and $c = \frac 12$. As above, there are eight cases. If $p \geq 7$ is prime then
\begin{align*}
a-1 = -\frac{4}{5} &= \begin{cases}
\overline{\left(\frac{4p-4}{5}\right)} & p \equiv 1 \pmod{5},\\
\overline{\left(\frac{p-4}{5}\right)\left(\frac{4p-1}{5}\right)} & p \equiv 4 \pmod{5},\\
\overline{\left(\frac{2p-4}{5}\right)\left(\frac{p-2}{5}\right)\left(\frac{3p-1}{5}\right)\left(\frac{4p-3}{5}\right)} & p \equiv 2 \pmod{5},\\
\overline{\left(\frac{3p-4}{5}\right)\left(\frac{p-3}{5}\right)\left(\frac{2p-1}{5}\right)\left(\frac{4p-2}{5}\right)} & p \equiv 3 \pmod{5},\\
\end{cases}\\
b-1 = -\frac{2}{3} &= \begin{cases}
\overline{\left(\frac{2p-2}{3}\right)}&p \equiv 1 \pmod{3},\\
\overline{\left(\frac{p-2}{3}\right)\left(\frac{2p-1}{3}\right)}&p \equiv 2 \pmod{3},
\end{cases}\\
c-1 = -\frac 12 &= \overline{\left(\frac{p-1}{2}\right)}.
\end{align*}
It is straightforward to check that for prime $p \geq 7$, the coefficients of $_2F_1(\frac 15,\frac 13;\frac 12;z)$ are $p$-adically unbounded if and only if $p \equiv 2,8$ or $14\pmod{15}$. For the remaining primes $p \geq 7$, the coefficients are in fact $p$-integral. Thus, the set of primes such that $_2F_1(\frac 15,\frac 13;\frac 12;z)$ has $p$-adically unbounded coefficients has a Dirichlet density of $\frac 38$.
\end{ex}

\begin{dfn}
\label{d:sgp}
Let $(a,b;c)$ denote rational hypergeometric parameters such that $c$ is not a negative integer. Then let $S(a,b;c)$ denote the set of primes $p$ such that $_2F_1(a,b;c;z)$ has $p$-adically unbounded coefficients.
\end{dfn}
As an application of Theorem \ref{t:necandsuff}, we show that the set $S(a,b;c)$ of unbounded primes for some admissible $_2F_1$ always has a Dirichlet density.
\begin{prop}
\label{p:SGP}
Let $(a,b;c)$ denote admissible hypergeometric parameters, and let $D$ denote the least common multiple of the denominators of $a$, $b$ and $c$. If $p > D$ is a good prime that satisfies $p \in S(a,b;c)$, then for all primes $q \geq p$ such that $q \equiv p \pmod{D}$, necessarily $q \in S(a,b;c)$ too. Thus $S(a,b;c)$ has a Dirichlet density of the form $\frac{\alpha}{\phi(D)}$ for an integer $\alpha$ satisfying $0 \leq \alpha \leq \phi(D)$, where $\phi(D)$ denotes Euler's $\phi$-function.
\end{prop}
\begin{proof}
Let $a_j(p)$ denote the $j$th $p$-adic digit of $a-1$, and define $b_j(p)$ and $c_j(p)$ similarly. Then by Lemma \ref{l:distinctdigits}, if $p > D$ we have $a_j(p) \neq c_j(p)$ and $b_j(p) \neq c_j(p)$ for all $j$. Hence by Theorem \ref{t:necandsuff}, to determine whether such a prime lies in $S(a,b;c)$, we need only determine whether there exists an index $j$ such that $c_j(p) > a_j(p)$ and $c_j(p) > b_j(p)$. If $M$ is the period of this data, then by periodicity of the $p$-adic expansions, we can concentrate on those $j$ in the range $0 \leq j < M$.

Let $p > D$ be a good prime such that $c_j(p) > a_j(p)$ for some index $0 \leq j < M$. By Lemma \ref{l:digits}, this is equivalent with
\begin{equation}
\label{eq:sgpineq}
  \floor{\left\{-p^{M-1-j}c\right\}p} > \floor{\left\{-p^{M-1-j}a\right\}p}.
\end{equation}
Let $q = p+tD$ denote another prime, where $t$ is a positive integer. Then the condition that $c_j(q) > a_j(q)$ is equivalent with
\[
  \floor{\left\{-p^{M-1-j}c\right\}p} + \left\{-p^{M-1-j}c\right\}tD > \floor{\left\{-p^{M-1-j}a\right\}p} + \left\{-p^{M-1-j}a\right\}tD.
\]
But this inequality is implied by \eqref{eq:sgpineq}. Hence if $c_j(p) > a_j(p)$, then $c_j(q) > a_j(q)$ for all primes $q \geq p$ such that $q\equiv p \pmod{D}$. Since the same argument holds with $a$ replaced by $b$, this concludes the proof of the Proposition.
\end{proof}
\begin{rmk}
\label{r:computedensity} There exists a simple algorithm for computing the Dirichlet density of $S(a,b;c)$ for any admissible parameters. One can simply use Theorem \ref{t:necandsuff} to check primes $p > D$ lying in the $\phi(D)$ possible congruence classes mod $D$. Once a prime is found such that $_2F_1(a,b;c;z)$ has unbounded $p$-adic coefficients, then the rest of the primes $q \geq p$ in the arithmetic progression $q \equiv p \pmod{D}$ are contained in $S(a,b;c)$ by Proposition \ref{p:SGP}. Lemma \ref{l:digits} can be used to provide a stopping criterion to determine if a congruence class has finite intersection with $S(a,b;c)$. For example, if $x \in (0,1)$ is rational, and $x-1 \in \bbZ_p^\times$, then by Lemma \ref{l:digits}, the $j$th $p$-adic digit $x_j(p)$ of $x-1$ satisfies
\[
  \{-p^{M-1-j}x\}-\frac 1p < \frac{x_j(p)}{p} < \{-p^{M-1-j}x\}.
\]
Observe that the fractional part above only depends on $p$ modulo the denominator of $x$. By the proof of Lemma \ref{l:distinctdigits} and since $p > D$, if we consider the corresponding fractional parts with $x = a,b,c$, then they are distinct. By periodicity of the $p$-adic expanions of $a-1$, $b-1$ and $c-1$ for admissible parameters, we need only check a finite number of coefficients using Theorem \ref{t:necandsuff}. Hence if $p$ is large enough to ensure that
\[
\frac 1p < \min_{0\leq j <M}\left(\min_{\substack{x,y \in \{a,b,c\}\\ x\neq y}}\left |\{-p^{M-1-j}x\}-\{-p^{M-1-j}y\}\right|\right),
\] 
then it suffices to test whether $p \in S(a,b,c)$ for such a prime. If $p \in S(a,b,c)$, then all primes $q \geq p$ satisfying $q \equiv p \pmod{D}$ will be contained in $S(a,b;c)$. Otherwise, $S(a,b;c)$ has finite intersection with this congruence class.
\end{rmk}

\begin{thm}
\label{t:density0}
Let $(a,b;c)$ denote admissible hypergeometric parameters. Let $D$ denote the least common multiple of the denominators of $a$, $b$ and $c$. Then the following are equivalent:
\begin{enumerate}
\item[(i)] the monodromy group of the corresponding hypergeometric differential equation is finite;
\item[(ii)] the set $S(a,b;c)$ is finite;
\item[(iii)] for every integer $u$ coprime to $D$, the fractional parts $\{ua\}$, $\{ub\}$ and $\{uc\}$ are such that $\{uc\}$ lies between $\{ua\}$ and $\{ub\}$.
\end{enumerate}
\end{thm}
\begin{proof}
For the equivalence of (i) and (iii), see Theorem 4.8 in \cite{BeukersHeckman}. It is well-known that (i) implies (ii), say by Eisenstein's theorem (see \cite{DworkVanDerPoorten} for an interesting discussion of this result). To complete the proof we will show that (ii) implies (iii).

Thus assume that (ii) holds. Let $a_j(p)$ denote the $p$-adic digits of $a-1$, and definte $b_j(p)$ and $c_j(p)$ similarly. By Theorem \ref{t:necandsuff} and Lemma \ref{l:distinctdigits} there exists a finite set of primes $S$ with the property that for each prime $p \not \in S$, and for every index $j$, one has either $c_j(p) < a_j(p)$ or $c_j(p) < b_j(p)$. By Lemma \ref{l:digits} we have either $\floor{\{-p^{M-1-j}c\}p} < \floor{\{-p^{M-1-j}a\}p}$ or $\floor{\{-p^{M-1-j}c\}p} < \floor{\{-p^{M-1-j}b\}p}$ for each index $j$.

Since $\{ua\}$ only takes values of the form $\frac{\alpha}{D}$, and similarly with $a$ replaced by $b$ and $c$, we see that for all but finitely many primes, in fact for each $j$ either $\{-p^{M-1-j}c\} < \{-p^{M-1-j}a\}$ or $\{-p^{M-1-j}c\} < \{-p^{M-1-j}b\}$. By varying $p$ and $j$ we obtain that for every $u$ coprime to $D$, where $D$ is the least common multiple of the denominators of $a$, $b$ and $c$, that either $\{uc\} < \{ua\}$ or $\{uc\} < \{ub\}$ (here we have used Dirichlet's theorem on primes in arithmetic progressions). 

Suppose that $\{uc\} < \{ua\}$ and $\{uc\} < \{ub\}$ for some integer $u$ coprime to $D$. Write $Duc = x_cD+r_c$, $Dub = x_bD+r_b$ and $Dua = x_aD+r_a$, where the remainders $r$ satisfy $0 < r < D$. Then $\frac{r_c}{D} < \frac{r_a}{D}$ and $\frac{r_c}{D} < \frac{r_b}{D}$. But then observe that
\[
  -Duc = -x_cD-r_c = (1-x_c)D+(D-r_c),
\]
and similarly for $-Dua$ and $-Dub$. Then it follows that $\{-uc\} > \{-ua\}$ and $\{-uc\} > \{-ub\}$, a contradiction. Hence it must be the case that for every integer $u$ coprime to the denominators, $\{uc\}$ lies between $\{ua\}$ and $\{ub\}$, as claimed.
\end{proof}

\begin{thm}
\label{t:density1}
Let $(a,b;c)$ denote admissible hypergeometric parameters, and let $D$ denote the least common multiple of the denominators of $a$, $b$ and $c$. Then the following are equivalent:
\begin{enumerate}
\item[(i)] $c$ is the smallest of the three parameters;
\item[(ii)] $S(a,b;c)$ contains all but finitely many primes;
\item[(iii)] $S(a,b;c)$ contains infinitely many primes $p$ such that $p \equiv 1 \pmod{D}$.
\end{enumerate}
In particular, one third of all hypergeometric series with admissible parameters have the property that their coefficients are $p$-adically unbounded for one-hundred percent of all primes.
\end{thm}
\begin{proof}
  First suppose that $c < a$ and $c < b$. Let $M$ denote the period of this data, and let $a_j(p)$, $b_j(p)$ and $c_j(p)$ denote the $p$-adic digits of $a-1$, $b-1$ and $c-1$, respectively. Then by Lemma \ref{l:digits} we have $a_{M-1}(p) =p +\floor{-pa}$, $b_{M-1}(p) = p+\floor{-pb}$ and $c_{M-1}(p) = p+\floor{-pc}$. Since $c<a$ and $c<b$, if $p > \max(\frac{1}{a-c},\frac{1}{b-c})$ we have $-a+\frac{1}{p} < -c$ and $-b+\frac{1}{p} < -c$. But this implies that $a_{M-1}(p) < c_{M-1}(p)$ and $b_{M-1}(p) < c_{M-1}(p)$, and hence $p \in S(a,b;p)$ by Theorem \ref{t:necandsuff}. Thus, if $c$ is the smallest of the three parameters, then $S(a,b;c)$ contains every prime $p$ satisfying $p > \max(\frac{1}{a-c},\frac{1}{b-c})$. That is, (i) implies (ii). That (ii) implies (iii) is obvious.

Finally suppose that $S(a,b;c)$ contains infinitely many primes $p$ of the form $p \equiv 1 \pmod{D}$. For such primes we have $M =1$ by Lemma \ref{l:periodicexpansion}, and thus by Theorem \ref{t:necandsuff} and Lemma \ref{l:digits} there exists a prime $p \equiv 1 \pmod{D}$ such that $\{-c\} > \{-a\}$ and $\{-c\} > \{-b\}$. Hence $c < a$ and $c < b$, which shows that (iii) implies (i).
\end{proof}

\section{Hypergeometric series and modular forms}
\label{s:Schwarzlist}

Let $\Gamma \subseteq \SL_2(\bbR)$ denote a Fuchsian group, let $Y$ denote the curve $\Gamma \backslash \uhp$, and let $Y'$ denote $Y$ with all elliptic points removed. If the image of $\Gamma$ in $\PSL_2(\bbR)$ can be generated by two elements, then $Y'$ can be identified with $\bbP^1\setminus \{0,1,\infty\}$. Let $u(\tau) \colon \uhp' \to Y'$ denote a uniformizing map realizing this isomorphism, where $\uhp'$ denotes $\uhp$ deprived of its elliptic points for $\Gamma$. This uniformization identifies the image of $\Gamma$ in $\PSL_2(\bbR)$ with the orbifold fundamental group of $Y$. This group is in general a quotient of the fundamental group of the Riemann surface $Y'$. Solutions to Fuchsian differential equations on $Y'$ can be pulled back via $u(\tau)$ to vector-valued modular forms (of weight $0$) of $\Gamma$ on $\uhp$ that transform according to the monodromy representation $\rho$ of the fundamental group of $Y'$, provided  $\rho$  factors through the orbifold fundamental group of $Y$. That is, they are vector-valued functions $F \colon \uhp \to \bbC^d$, meromorphic at the cusps and elliptic points and holomorphic elsewhere, that satisfy a transformation law
\[
F(\gamma \tau) {=} \rho(\gamma)F(\tau)
\]
for all $\tau {\in} \uhp$ and $\gamma {\in} \Gamma$. This relation of vector-valued modular forms and Fuchsian differential equations on $\bbP^1$ goes back at least to \cite{BantayGannon}.

A natural case to consider is the group $\Gamma(2)$, whose image in $\PSL_2(\bbR)$ is free on two generators. Vector-valued modular forms for $\Gamma(2)$ thus describe all solutions of Fuchsian equations on $\bbP^1\setminus\{0,1,\infty\}$. If $\lambda(\tau)$ denotes a uniformizing map taking the cusps $0$, $1$ and $\infty$ of $\Gamma(2)$ to $0$, $1$ and $\infty$ in $\bbP^1$, respectively, then for each hypergeometric series $_nF_{n-1}(\alpha_i;\beta_j;z)$ that we've been considering, the function $_nF_{n-1}(\alpha_i;\beta_j;\lambda(\tau))$ is a component of a vector-valued modular form for some $n$-dimensional representation $\rho$ of $\Gamma(2)$. Conversely, all (weakly-holomorphic) vector-valued modular forms for $\Gamma(2)$ of weight $0$ can be expressed in the form $F(\lambda(\tau))$ where $F(z)$ is a vector whose entries form a basis of solutions of a Fuchsian differential equation on $\bbP^1$. In general the function $F(z)$ need not arise from a generalized hypergeometric differential equation. 

These observations connect the question of unbounded denominators of Taylor coefficients of solutions of differential equations with the question of unbounded denominators of Fourier coefficients of modular forms. Note, though, that $F(z)$ and $F(\lambda(\tau))$ need not have the exact same set of primes $p$ such that their coefficients are $p$-adically bounded. The difference between these two sets of primes is a finite set. For example, a common occurrence is for the modular form $F(\lambda(\tau))$ to have integer coefficients, say due to it being a congruence modular form, whereas $F(z)$ could have a finite number of unbounded primes occurring in its coefficients that are cancelled upon substituting in the uniformizing map $\lambda(\tau)$. Richard Gottesman treats the question of unbounded denominators for vector-valued modular forms of rank $2$ on $\Gamma(2)$ in detail in his upcoming PhD thesis \cite{Gottesman}, using the ideas discussed above.

In \cite{FrancMason1}, Franc-Mason studied the somewhat simpler case of $\SL_2(\bbZ)$. Although $\SL_2(\bbZ)$ is two-generated, it is not free on two generators, and so not all solutions of Fuchsian equations can be described in terms of vector-valued modular forms for $\SL_2(\bbZ)$. Conversely, \cite{FrancMason1},\cite{Gannon} observed that all holomorphic modular forms for 
$\SL_2(\bbZ)$ of rank two can be described in terms of solutions of hypergeometric differential equations. That is, one need not consider more general Fuchsian equations on $\bbP^1\setminus \{0,1,\infty\}$ of rank two\footnote{Of course, the study of such equations can be reduced to the study of hypergeometric equations, as was known to Riemann.}. The case of $\SL_2(\bbZ)$ is simplified further by the work of Mason in \cite{Mason2-dim}, which shows that all finite image representations $\rho$ of $\SL_2(\bbZ)$ are such that $\ker \rho$ is a congruence subgroup. Thus, the question of unbounded denominators amounts to proving in this case that when the image of $\rho$ is not finite, then the corresponding modular forms have unbounded denominators in a strong sense: if $\rho$ has infinite image, then there should be infinitely many primes appearing in the denominators of modular forms for $\rho$. The paper \cite{FrancMason1} made very modest use of hypergeometric series to prove this assertion. Essentially \cite{FrancMason1} showed that for such a representation $\rho$, there exists a modular form for $\rho$ and an arithmetic progression of primes $p$ such that $p$ appears at least to power $p^{-1}$ in the coefficients of the modular form. Using this, unbounded denominators were then established for all modular forms associated with $\rho$. By Theorem \ref{t:density0} above, we now know that in the infinite image case, not only does $p$ appear at least once in some denominator, but in fact there must exist a positive density of primes $p$ such that a given modular form for $\rho$ has $p$-adically unbounded Fourier coefficients.

In the remainder of this section we collect data and facts from \cite{Mason2-dim} and \cite{FrancMason1} to describe a modular Schwartz list for $\SL_2(\bbZ)$. That is, we describe all finite-image irreducible representations of $\SL_2(\bbZ)$ of rank two, as well as the corresponding modular forms and hypegeometric series. Note that unlike the classical Schwartz list, which is infinite due to a proliferation of dihedral representations, this modular Schwarz list at level one is in fact finite, and it includes only congruence representations.

 Let $\rho : \SL_2(\bbZ) \to \GL_2(\bbC)$ be an irreducible, 2-dimensional representation of $\SL_2(\bbZ)$ with \emph{finite image} and let $F(\tau)$ be a nonzero holomorphic vector-valued modular form of least integral weight $k_0$ for $\rho$. Thus $F : \uhp \to \bbC^2$ is holomorphic and satisfies
\[
F(\gamma\tau) = (c\tau+d)^{k_0}\rho(\gamma)F(\tau) \quad  \textrm{for all} \quad \gamma = \twomat abcd \in \SL_2(\bbZ).
\]
Unfortunately, unless $k_0 = 0$, this modular form is a section of a projectively flat holomorphic vector bundle that is not flat. Since the square $\eta^2$ of Dedekind's $\eta$-function transforms under $\SL_2(\bbZ)$ via a character $\chi$, and since it is nonvanishing in $\uhp$, we can set $\hat F(\tau) = F(\tau)/\eta^{2k_0}$ to shift $F$ to weight zero, but at the expense of changing the representation $\rho$ to $\hat \rho = \rho \otimes \chi^{-k_0}$. This adjusted function is naturally a global section of a holomorphic connection with a regular singularity at the cusp of $\SL_2(\bbZ)$, and so it thus satisfies an ordinary differential equation. More precisely, the paper \cite{FrancMason1} showed that the component functions of $\hat{F}(\tau)$ are a pair of fundamental solutions of a hypergeometric differential equation which has $\hat{\rho}$ as its monodromy representation. In particular, the components of $F(\tau)$ may be expressed in terms of hypergeometric series evaluated at a certain level one hauptmodul as follows:
 \begin{eqnarray}\label{2components}
&&f_1(\tau){:=}\eta^{2k_0}(\tau)j^{-a}(\tau) {_2}F_1(a, 1{+}a{-}c; 1{+}a{-}b; J^{-1})\\
&&f_2(\tau){:=}\eta^{2k_0}(\tau)j^{-b}(\tau) {_2}F_1(b, 1{+}b{-}c; 1{+}b{-}a; J^{-1})\notag
\end{eqnarray}
for certain constants $a, b, c$ (see below), and where $q{:=}e^{2\pi i\tau}$ and
\begin{eqnarray*}
&&j(\tau){:=}\frac{E_4^3(\tau){-}E_6^2(\tau)}{\Delta(\tau)},\ \ \ J(\tau){:=}\frac{j(\tau)}{1728},
\ \ \ \eta(\tau){:=}q^{1/24}\prod_{n=1}^{\infty} (1-q^n).
\end{eqnarray*}
Note that these expressions depend on a choice of basis for $\rho$. In particular, since $\rho$ is of finite image and irreducible, the matrix $\rho(T)$ (where $T = \stwomat 1101$) has distinct roots of unity as eigenvalues, and so we diagonalize it as
\[
\rho(T)=\left(\begin{array}{cc} e^{2\pi i m_1} & 0 \\0 & e^{2\pi im_2}\end{array}\right)
\]
for rational numbers $m_1,m_2 \in [0,1)$. (Incidentally, $\rho \stwomat 0{-1}10$ is computed for all 2-dimensional examples, in this basis \eqref{2components}, in Section 4.2 of \cite{Gannon}.) With this notation, one finds (as in \cite{FrancMason1}) that
\begin{align*}
a &= \frac{1}{12} + \frac{m_1-m_2}{2}, & b &= \frac{1}{12} - \frac{m_1-m_2}{2}, & c &= \frac{2}{3}. 
\end{align*}
Note that in certain places in \cite{Mason2-dim} and \cite{FrancMason1} it was convenient to assume that $m_1 \leq m_2$, but it is not necessary to do so for these formulae to hold, and so we make no such hypothesis here. Note also that the formulae in (\ref{2components}) arise from solving a hypergeometric equation at the singular point $\infty$, and this is why the quantities $a$, $1+a-c$, etc occur, rather than $a$, $b$ and $c$.

Since $\rho$ is an irreducible representation of $\SL_2(\bbZ)$ of finite image,  it is also known \cite{Mason2-dim} that $\rho$ has a congruence subgroup for its kernel, say of level $N$. Therefore, the $q$-series expansions of the components of $F(\tau)$ are classical scalar modular forms of level $N$. They thus have bounded denominators, but this does not mean that the hypergeometric series in \eqref{2components} necessarily have bounded denominators (and indeed, they do not).

Below we list the $54$ isomorphism classes of irreducible and finite image $\rho$ together with relevant data pertaining to both the representation and the corresponding hypergeometric differential equation. We will see that, in a sense, only $18$ different hypergeometric series are involved. The data is organized so that each table corresponds to one orbit of representations under tensoring with the one-dimensional characters of $\SL_2(\bbZ)$. There are five distinct orbits, four containing twelve representations each, and one containing only six representations. To see that the dihedral orbit is indeed only of size six, one must use the fact that irreducible representations of $\SL_2(\bbZ)$ of dimension two are determined up to isomorphism by their $\rho(T)$ eigenvalues, and that swapping the eigenvalues yields an isomorphic representation (see \cite{Mason2-dim}). Finally, since we have $c = \frac{2}{3}$ in all of these examples, we omit $c$ from the data. 
\bigskip

\begin{minipage}{\textwidth}
    \centering
    \begin{tabular}{|ll|l|l|ll|}
      \hline
$m_1$ & $m_2$ & $N$ & $k_0$ & $a$ & $b$ \\
\hline
0&1/2& 2& 2 & -1/6&1/3\\
1/12&7/12& 12& 3&-1/6& 1/3\\
1/6&2/3& 6& 4&-1/6& 1/3\\
1/4&3/4& 4& 5 &-1/6& 1/3\\
1/3&5/6& 6& 6 &-1/6& 1/3\\
      5/12&11/12 &12& 7 &-1/6& 1/3\\
                               \hline
    \end{tabular}
    \medskip

\emph{Table 1:} Dihedral orbit.
\end{minipage}
\medskip

\begin{minipage}{\textwidth}
    \centering
\begin{tabular}{|ll|l|l|ll|}
\hline
  $m_1$ & $m_2$ & $N$ & $k_0$ & $a$ & $b$ \\
\hline
0 &2/3 & 3 & 3 & -1/4 & 5/12 \\
1/12 &3/4 & 12 & 4 & -1/4 & 5/12 \\
1/6 &5/6  & 6 & 5 & -1/4 & 5/12 \\
  1/4 &11/12 & 12 & 6 & -1/4 & 5/12 \\
  \hline
1/3 & 0 & 3 & 1 & 1/4 & -1/12 \\
5/12 & 1/12 & 12 & 2 & 1/4 &  -1/12 \\
1/2 & 1/6 & 6 & 3 & 1/4 & -1/12 \\
7/12 & 1/4 & 12 & 4 & 1/4 &  -1/12 \\
2/3 & 1/3 & 3 & 5 & 1/4 & -1/12 \\
3/4 & 5/12 & 12 & 6 & 1/4 & -1/12 \\
5/6 & 1/2 & 6 & 7 & 1/4 &  -1/12 \\
  11/12 & 7/12 & 12 & 8 & 1/4 & -1/12 \\
  \hline
\end{tabular}
    \medskip

\emph{Table 2:} Tetrahedral orbit.
\end{minipage}
\bigskip

  \begin{minipage}{\textwidth}
    \centering
  \begin{tabular}{|ll|l|l|ll|}
    \hline
$m_1$ & $m_2$ & $N$ & $k_0$ & $a$ & $b$ \\
\hline
1/24 &19/24& 24 & 4& -7/24 &11/24\\
1/8 &7/8 & 8& 5& -7/24 & 11/24 \\
    5/24 &23/24& 24 & 6 & -7/24 &11/24 \\
    \hline
7/24 & 1/24& 24& 1& 5/24&-1/24\\
3/8 &  1/8 & 8& 2& 5/24&-1/24\\
11/24 & 5/24 & 24 & 3 & 5/24&-1/24\\
13/24 & 7/24 & 24 & 4 & 5/24&-1/24 \\
5/8 & 3/8 & 8 & 5 & 5/24&-1/24 \\
17/24 & 11/24 & 24 & 6 & 5/24&-1/24 \\
19/24 & 13/24 & 24 & 7 & 5/24&-1/24\\
7/8 & 5/8 & 8 & 8 & 5/24&-1/24 \\
    23/24 & 17/24 & 24 & 9 & 5/24&-1/24\\
\hline
  \end{tabular}
  \medskip

\emph{Table 3:} Octahedral orbit.
\end{minipage}
\bigskip

\begin{minipage}{\textwidth}
  \centering
\begin{tabular}{|ll|l|l|ll|}
\hline
  $m_1$ & $m_2$ & $N$ & $k_0$ & $a$ & $b$ \\
\hline
1/30 & 19/30 & 30 & 3 & -13/60 & 23/60 \\
7/60 & 43/60 & 60 & 4 & -13/60 & 23/60 \\
1/5 & 4/5 & 5 & 5 & -13/60 & 23/60 \\
17/60 & 53/60 & 24 & 6 & -13/60 & 23/60 \\
  11/30 & 29/30 & 30 & 7 & -13/60 & 23/60 \\
  \hline
9/20 & 1/20 & 20 & 2 & 17/60 & -7/60 \\
8/15 & 2/15 & 15 & 3 & 17/60 & -7/60 \\
37/60 & 13/60 & 60 & 4 & 17/60 & -7/60 \\
7/10 & 3/10 & 10 & 5 & 17/60 & -7/60 \\
47/60 & 23/60 & 60 & 6 & 17/60 & -7/60 \\
13/15 & 7/15 & 15 & 7 & 17/60 & -7/60 \\
  19/20 & 11/20 & 20 & 8 & 17/60 &  -7/60 \\
  \hline
\end{tabular}
\medskip

\emph{Table 4:} Icosahedral orbit 1.
\end{minipage}
\bigskip

\begin{minipage}{\textwidth}
\centering
\begin{tabular}{|ll|l|l|ll|}
  \hline
$m_1$ & $m_2$ & $N$ & $k_0$ & $a$ & $b$ \\
\hline
1/60 & 49/60 & 60 & 4 & -19/60 & 29/60 \\
1/10 & 9/10 & 10 & 5 & -19/60 &29/60 \\
  11/60 & 59/60 & 60 & 6 & -19/60 & 29/60 \\
  \hline
4/15 & 1/15 & 15 & 1 & 11/60 & -1/60 \\
7/20 & 3/20 & 20 & 2 & 11/60 & -1/60 \\
13/30 & 7/30 & 30 & 3 & 11/60 & -1/60 \\
31/60 & 19/60 & 60 & 4 & 11/60 & -1/60 \\
3/5 & 2/5 & 5 & 5 & 11/60 & -1/60 \\
41/60 & 29/60 & 60 & 6 & 11/60 & -1/60 \\
23/30 & 17/30 & 30 & 7 & 11/60 & -1/60 \\
17/20 & 13/20 & 20 & 8 & 11/60 & -1/60 \\
  14/15 & 11/15 & 15 & 9 & 11/60 & -1/60 \\
  \hline
\end{tabular}
\medskip

\emph{Table 5:} Icosahedral orbit 2.
\end{minipage}
\bigskip

For a discussion about why the parameters change when the eigenvalue $e^{2\pi im_2}$ wraps around the circle, see Remark 3.12 of \cite{CandeloriFranc}. We see that there are nine \emph{essential} pairs of hypergeometric series
\begin{align*}
{_2}F_1\left(a,\frac{1}{3}+a;1+a-b;z\right) && {_2}F_1\left(b,\frac{1}{3}+b;1+b-a;z\right)
\end{align*}
that play a r\^{o}le in the theory of holomorphic vector-valued modular forms for two-dimensional irreducible representations of $\SL_2(\bbZ)$ of finite image. The results of the present paper allow one to easily check that these series have $p$-adically unbounded coefficients for only finitely many primes $p$. A more careful analysis at the unbounded primes allows one to show that after substituting $z=J^{-1}$ and multiplying by the appropriate power of $J$, one obtains $q$-series with rational coefficients having bounded denominators, as one knows.

\bibliographystyle{plain}

\end{document}